\documentclass[a4paper,11pt]{article}

\usepackage{amsmath, latexsym}
\usepackage{amssymb,amsthm,mathtools,xspace}
\usepackage{amscd,graphics}
\usepackage{enumerate,mathrsfs,url}
\usepackage[usenames,dvipsnames]{color}
\usepackage[normalem]{ulem}

\usepackage[colorlinks=false, pdfborder={0 0 0}]{hyperref}

\numberwithin{equation}{section}

\newtheorem{prop}{Proposition}[section]
\newtheorem{lem}[prop]{Lemma}

\newtheorem{thm}[prop]{Theorem}

\newtheorem{cor}[prop]{Corollary}

\theoremstyle{remark}

\newtheorem{rem}[prop]{Remark}

\theoremstyle{definition}
\newtheorem{Def}[prop]{Definition}

\newtheorem{exa}[prop]{Example}

\newcommand{\rrlc}{rrl-continuation\xspace}
\newcommand{\rrlci}{\textit{rrl-continuation}\xspace}
\newcommand{\rrlcs}{rrl-continuations\xspace}
\newcommand{\rrlce}{rrl-continuable\xspace}
\newcommand{\rrlcy}{rrl-continuability\xspace}

\newcommand{\beglabel}[1]{\begin{equation}	\label{#1}}
\newcommand{\elabel}{\end{equation}}

\newcommand{\eps}{\varepsilon}     
\newcommand{\defeq}{\coloneqq} 

\newcommand{\C}{\mathbb{C}}       
\newcommand{\N}{\mathbb{N}}  
\newcommand{\Q}{\mathbb{Q}}       
\newcommand{\R}{\mathbb{R}}       
\newcommand{\Z}{\mathbb{Z}}       
\newcommand{\B}{\mathbb{B}}
\newcommand{\D}{\mathbb{D}}
\newcommand{\E}{\mathbb{E}}

\renewcommand{\S}{\mathbb{S}}
\newcommand{\T}{\mathbb{T}}

\newcommand{\PP}{\widehat{\mathbb{C}}}

\newcommand{\ph}{\varphi}          

\newcommand{\be}{\beta}
\newcommand{\de}{\delta}
\newcommand{\De}{\Delta}
\newcommand{\ga}{\gamma}
\newcommand{\Ga}{\Gamma}

\newcommand{\la}{\lambda}
\newcommand{\La}{\Lambda}

\renewcommand{\th}{\theta}
\newcommand{\Sig}{\Sigma}

\newcommand{\om}{\omega}

\newcommand{\cN}{\mathcal{N}}

\newcommand{\gC}{\mathscr C}       
\newcommand{\gH}{\mathscr H}       

\newcommand{\gO}{\mathscr O}
\newcommand{\gM}{\mathscr M}

\newcommand{\col}{\colon}          

\newcommand{\I}{{\mathrm i}}
\newcommand{\dd}{{\mathrm d}}      
\newcommand{\ee}{{\mathrm e}}      
\newcommand{\pa}{\partial}

\newcommand{\ii}{^{-1}}
\newcommand{\demi}{\frac{1}{2}}
\newcommand{\ti}{\tilde}
\newcommand{\ens}{\enspace}

\newcommand{\ie}{\emph{i.e.}\@\xspace}
\newcommand{\eg}{\emph{e.g.}\@\xspace}

\newcommand{\rhs}{{right-hand side}\xspace}
\newcommand{\suppar}[1]{^{(#1)}}

\DeclareMathOperator{\dist}{dist}     

\newcommand{\dst}{\displaystyle}

\newcommand{\wrt}{with respect to}

\newcommand{\ov}{\overline}

\DeclarePairedDelimiter\abs{\lvert}{\rvert}%
\DeclarePairedDelimiter\norm{\lVert}{\rVert}%
\makeatletter
\let\oldabs\abs
\def\abs{\@ifstar{\oldabs}{\oldabs*}}
\let\oldnorm\norm
\def\norm{\@ifstar{\oldnorm}{\oldnorm*}}
\makeatother
\newcommand{\IN}{^{\textnormal{(i)}}}
\newcommand{\EX}{^{\textnormal{(e)}}}

\newcommand{\un}[1]{{\underline{#1}}}

\DeclareMathOperator{\supp}{\rule[-.3\baselineskip]{0pt}{.8\baselineskip}supp}

\newcommand{\elSC}{\ell^1(\S,\C)}
\newcommand{\elSB}{\ell^1(\S,B)}
\newcommand{\dlSB}{\ell^{1/2}(\S,B)}
\newcommand{\qlSB}{\ell^{1/4}(\S,B)}
\newcommand{\enSB}{\ell^\nu(\S,B)}

\newcommand{\ministrut}{\rule[-.25\baselineskip]{0pt}{.1\baselineskip}}

\newcommand{\xHecke}{x^*_{\textrm H}}
\newcommand{\gHecke}{g_{\textrm H}}
\newcommand{\gHe}[1]{g_{{\textrm H},#1}}

\begin{document}

\title{Generalised continuation by means of right limits}

\author{David Sauzin and Giulio Tiozzo}


\maketitle

\vspace{-.75cm}

\begin{abstract} 
Several theories have been proposed to generalise the concept of analytic continuation
to holomorphic functions of the disc for which the circle is a natural boundary.
Elaborating on Breuer-Simon's work on \emph{right limits} of power series,
Baladi-Marmi-Sauzin recently introduced the notion of \emph{renascent right
limit} and \rrlci. 

We discuss a few examples and consider particularly the classical example of
\emph{Poincar\'e simple pole series} in this light. These functions are
represented in the disc as series of infinitely many simple poles located on the
circle;
they appear for instance in small divisor problems in dynamics. 
We prove that any such function admits a unique \rrlc, which coincides with the
function obtained outside the disc by summing the simple pole expansion.
We also discuss the relation with monogenic regularity in the sense of
Borel.
\end{abstract}


\setcounter{tocdepth}{2}
\tableofcontents

\newpage

\section{Introduction}

When one is given a function~$g$ holomorphic in the unit disc~$\D$,
one can ask whether~$g$ is related in some way to a holomorphic function
defined outside the disc.
A first answer to the question comes from Weierstrass's notion of analytic
continuation.  Given a point~$\la$ on the unit circle, if there exists a
neighbourhood~$V$ and a holomorphic function on~$V$ whose restriction to $V \cap
\D$ is~$g$, then we say that $\la$ is a \emph{regular point} and the restriction
of~$g$ to the outer part of~$V$ is an analytic continuation.
%
%
%
If there is no regular point on the unit circle, then we say that the unit
circle is a \emph{natural boundary} for~$g$, but is it the end of the story?

It is the purpose of ``generalised analytic continuation'' to investigate this
situation and suggest other ways in which an outer function can be related to
the inner function~$g$.
The reader is referred to the monograph \cite{RS} for a panorama of various
theories which have been proposed to go beyond Weierstrass's point of view on
analytic continuation.

In this paper, we shall explore a new type 
of generalised analytic
continuation, called \rrlci, which is based on the notion 
of \emph{right limits} introduced in \cite{BrSi} as a
tool unifying various classical criteria to detect a natural boundary (an earlier, related
approach is due to \cite{Ag}).
The notion of \rrlc 
was put forward in the recent article
\cite{BMS} to deal with the natural boundary of Ruelle's susceptibility function; 
the goal of this paper is to start developing a general theory 
of \rrlc, see how it applies to some classical cases and compare
it to other theories of analytic continuation.

%









We start by giving a few definitions and analysing 
some basic examples.
In particular, we shall define the class of \emph{rrl-continuable} functions 
as functions on the unit disc whose Taylor expansion at the origin admits a renascent right limit
 (see section \ref{S:rrlc}). Each renascent right limit determines a function on the complement of $\mathbb{D}$ 
 which we call an \rrlci of $g$.

Let now $g$ be an rrl-continuable function:
if $g$ is continuable outside~$\D$ in the traditional sense, then
its analytic continuation across any arc coincides with its \rrlc (and thus must be unique); 
on the other hand, if $g$ has more than one rrl-continuation,
then the unit circle must be a natural boundary for $g$ (Proposition \ref{proprrlalternative}).
Particularly interesting is thus the case of functions called \emph{uniquely rrl-continuable},
that admit a unique \rrlc: indeed, in this case $\mathbb{D}$
may or may not be a natural boundary, but the \rrlc 
is a canonically defined function outside the disc which we 
may think of as continuation of $g$.

In the present paper we construct uniquely rrl-continuable functions 
in several contexts and explore the relationship between the function and 
its rrl-continuation(s).

 





As we shall see, the notion of \rrlc is well suited to study 
 power series generated by dynamical systems; in particular, we  analyse in this light
Hecke's example $\gHecke(z) = \sum_{k\ge1} \{k\th\}
z^k$, where $\th\in \R\setminus\Q$ and $\{\,\cdot\,\}$ denotes the
fractional part function.
It was shown in \cite{BrSi} how the theory of right limits implies
that~$\gHecke$ has a natural boundary on the unit circle; we show that
it has a unique \rrlc given by
$-\sum_{n<0} \{n\th\} z^n = \gHecke(z\ii) + (1-z)\ii$ for $\abs{z}>1$.


In the second part of the paper, we apply the theory to the classical situation, first
considered by Poincar\'e in 1883, where $g(z)$ is defined for $\abs{z}<1$ as a
series of simple poles
\[ g(z) \defeq \sum_{n \ge 0} \frac{\rho_n}{z-\ee^{\I\th_n}} \]
where the points $\ee^{\I\th_n}$ are dense on the unit circle
and the nonzero complex numbers~$\rho_n$ form an absolutely
convergent series.
For such a function (which we call \emph{Poincar\'e simple pole
  series}, or \emph{PSP-series} for short) the unit circle is a
natural boundary in the classical sense; however, there is a natural
candidate for the outer function, namely the sum $h(z)$ of the simple
pole series for $\abs{z}>1$.
We will prove that every simple pole 
series $g(z)$ inside the disc has a unique \rrlc, which coincides with 
the outer function: 
\begin{thm}    \label{thmPSPuniquerrlce}	
Let $g\in\gO(\D)$ be an inner PSP-series. Then $g$ is uniquely \rrlce and 
its \rrlc is the associated outer PSP-series.
\end{thm}
This result was announced in \cite[Appendix~A.2]{BMS}.
%
%
As an unexpected byproduct, we obtain that, for any
$\th\in\R\setminus\Q$, the function $g(z) = \sum_{k\ge0}
\dist(k\th,\Z) z^k$, which is somewhat similar to Hecke's example, has
a natural boundary on the unit circle and a unique \rrlc. 
PSP-series also appear in a dynamical context as 
solutions 
to the cohomological equation
for small divisor problems \cite{MS1} (see section \ref{S:mon}).
In fact, in this case, we have to consider PSP-series
with values in an arbitrary complex Banach space;
the theory of rrl-continuation still makes sense, and 
 we shall prove a more general version 
of Theorem \ref{thmPSPuniquerrlce} for vector-valued PSP-series 
(Theorem \ref{thm_BWDuniqueRRL}).

%


We will also compare the concept of rrl-continuability with Borel's
concept of monogenic regularity as developed in \cite{MS1},
\cite{MS2}: in particular, we shall see (Theorem~\ref{thmmonog})
that a large class of PSP-series is monogenic, and their continuation
in the sense of monogenic functions coincides with the outer series.
This fact raises the question whether any monogenic function admits a
unique rrl-continuation.


\section{Continuation by renascent right limits}	\label{secremindrrlc}

\subsection{Preliminaries}

We are interested in holomorphic functions defined in the unit disc by
power series of the form
\beglabel{eqdefgfroma}
g(z) = \sum_{k=0}^\infty a_k z^k
\elabel
with a bounded sequence of complex coefficients $(a_k)_{k\ge0}$.
Our aim is to investigate the possibility of defining ``generalised analytic
continuations'' for $\abs{z}>1$ when the unit circle is a natural boundary.
We will accept what is called a strong natural boundary in \cite{BrSi}:
\begin{Def}
A function $g(z)$ holomorphic in the unit disc is said to have a
\emph{strong natural boundary on the unit circle} if,
for every nonempty interval $(\om_1,\om_2)$,
\begin{equation}	\label{eqdefstrgnatb}
\sup_{0<r<1} \int_{\om_1}^{\om_2} \abs{g(r \ee^{\I \om})}\, \dd\om =\infty.
\end{equation}
\end{Def}
Clearly, if the unit circle is a strong natural boundary for~$g$, then the unit
circle is a natural boundary in the usual sense, since the function is not even
bounded in any sector $\{r \ee^{\I \om}\mid r \in (0,1),\,\om\in (\om_1,\om_2)\}$.


The article \cite{BrSi} provides a remarkable criterium to detect strong natural
boundaries (Theorem~\ref{thmBrSi} below), based on the notion of right limit
that we now recall.

\begin{Def}
\label{D:RL}
\begin{enumerate}[(i)]
\item
Let $\un a = (a_k)_{k\ge0}$ be a sequence in a topological space~$E$.
A \emph{right limit of~$\un a$} is any two-sided sequence $\un b = (
b_n)_{n\in\Z}$ of~$E$ for which there exists a strictly increasing sequence of positive
integers $(k_j)_{j\ge1}$ such that
\beglabel{eqlimankj}
\lim_{j\to\infty} a_{n+k_j} = b_n
\quad \text{for every $n\in\Z$}.
\elabel
\item
Let $g$ be a holomorphic function of the unit disc. We say that $\un b$ is a
\emph{right limit of~$g$} if the sequence~$\un a$ formed by the Taylor
coefficients at the origin, $a_k \defeq g\suppar k(0)/k!$, is bounded and $\un
b$ is a right limit of~$\un a$.
\end{enumerate}
\end{Def}


In view of~\eqref{eqlimankj}, each~$b_n$ must be an accumulation point of~$\un
a$. 
When $E$ is a compact metric space, every sequence~$\un a$ admits right limits;
given $\ell\in\Z$ and $c$ accumulation point of~$\un a$, one can always find
a right limit~$\un b$ such that $b_\ell = c$
(see \eg \cite[Lemma~2.1]{BMS}).

In the case of a function~$g$ with bounded Taylor coefficients, each right
limit gives rise to two generating series which will play an important role when
investigating the boundary behaviour of~$g$:

\begin{Def}	\label{definnout}
Given a two-sided bounded sequence of complex numbers $\un b = ( b_n)_{n\in\Z}$,
we define the \emph{inner and outer functions associated with~$\un b$} as
\begin{align*}
g_{\un b}^+(z) &= \sum_{n\ge0} b_n z^n, &\hspace{-4em} z&\in \D, \\[1ex]
g_{\un b}^-(z) &= -\sum_{n<0} b_n z^n, &\hspace{-4em} z&\in \E,
\end{align*}
where $\D = \{\, z\in\C \mid |z|< 1 \,\}$ is the unit disc and 
$\E = \{\, z\in\C \mid |z|>1 \,\} \cup\{\infty\}$ is a disc centred at~$\infty$
in the Riemann sphere~$\PP$.
\end{Def}

\begin{Def}
Given an arc~$J$ of the unit circle, $\un b = ( b_n)_{n\in\Z}$ is said to be
\emph{reflectionless on~$J$} if $g_{\un b}^+$ has an analytical
continuation in a neighbourhood~$U$ of~$J$ in~$\C$ and this analytical
continuation coincides with~$g_{\un b}^-$ on $U\cap \E$.
\end{Def}
Note that this terminology stems
from the spectral theory of Jacobi matrices, it is not related to the
Schwarz reflection principle.


\begin{thm}[Breuer--Simon, \cite{BrSi}]	\label{thmBrSi} 
Let $g$ be holomorphic in~$\D$ with bounded Taylor coefficients at~$0$.
\begin{enumerate}[(i)]
\item
Consider a nonempty interval $(\om_1,\om_2)$ and the corresponding arc of the
unit circle
$J = \{\, \ee^{\I\om} \mid \om\in(\om_1,\om_2) \,\}$,
and assume that \eqref{eqdefstrgnatb} is violated. Then every right limit
of~$g$ is reflectionless on~$J$.
\item
If $\un b = ( b_n)_{n\in\Z}$ and $\ti{\un b} = ( \ti b_n)_{n\in\Z}$ are two distinct
right limits of~$g$ and if there exists $N\in\Z$ such that either 
$b_n = \ti b_n$ for all $n\ge N$
or $b_n = \ti b_n$ for all $n\le N$,
then the unit circle is a strong natural boundary for~$g$.
\end{enumerate}
\end{thm}

\subsection{The \rrlce functions} \label{S:rrlc}

In \cite{BMS}, motivated by Breuer--Simon's work, right limits were
used to define a type of generalised analytic continuation as follows:
\begin{Def}
	\label{defrrl}
\begin{enumerate}[(i)]
\item
A \emph{renascent right limit} of a sequence~$\un a$ in a topological space is
any right limit $\un b = (b_n)_{n\in\Z}$ of~$\un a$ such that $b_n=a_n$ for
all $n\ge 0$.
\item
An \emph{\rrlce function} is a holomorphic function~$g$
which admits a renascent right limit~$\un b$;
then $g_{\un b}^+ = g$ in~$\D$ and the function~$g_{\un b}^-$, which is
holomorphic in~$\E$ and vanishes at~$\infty$, is said to be an \emph{\rrlc of~$g$}.
\item An \rrlce function~$g$ is said to be \emph{uniquely \rrlce} if it has a
unique \rrlc; in the opposite case, it is said to be \emph{polygenous}.
\end{enumerate}
\end{Def}

As a consequence of Theorem \ref{thmBrSi}, the situation for rrl-continuable 
functions is simpler than for arbitrary functions:

%
\begin{prop}	\label{proprrlalternative}
Let $g$ be an \rrlce function. Then
\begin{enumerate}[(i)]
\item
either there is an arc of the unit circle through which $g$ admits analytic
continuation; then $g$ is uniquely \rrlce and all the analytic continuations
of~$g$ through arcs of the unit circle match and coincide with the
\rrlc of~$g$;
\item
or the unit circle is a strong natural boundary for~$g$.

If $g$ is polygenous, then the unit circle is a strong natural boundary for~$g$.
\end{enumerate}
\end{prop}


\begin{proof}
Let $\un b$ be a renascent right-limit of~$g$.
Suppose there exists a closed arc
$J = \{\, \ee^{\I\om} \mid \om\in[\om_1,\om_2] \,\}$
in the neighbourhood of which~$g$ admits an analytic continuation~$h_J$.
Then
\[ \sup_{0<r<1} \int_{\om_1}^{\om_2} |g(r \ee^{\I \om})|\, \dd\om < \infty, \]
hence, by Theorem~\ref{thmBrSi}(i), $\un b$ is reflectionless on~$J$;
since $g = g_{\un b}^+$, this means that
$h_J = g_{\un b}^-$, independently of the choice of~$J$.
The uniqueness of the renascent right-limit follows too.

If on the contrary there is no analytic continuation for~$g$ across any arc
of the unit circle, then~$\un b$ is not reflectionless on any
arc thus Theorem~\ref{thmBrSi}(i) entails that the unit circle is a strong natural boundary.

The last statement follows from Theorem~\ref{thmBrSi}(ii).
\end{proof}

When an \rrlce function~$g$ has a natural boundary on the unit circle, we may
still think of the \rrlcs of~$g$ as being somewhat ``connected'' to~$g$ and consider
them as a kind of generalised analytic continuation, and the case of a unique
\rrlc may then be particularly interesting.%
%


\begin{exa}
In the case of a preperiodic sequence,
$g(z) = \sum_{k=0}^\infty a_k z^k$ with $a_k = a_{k+p}$ for all $k\ge m$, 
one checks easily that there is no renascent right limit unless $m=0$, \ie the
sequence is periodic, 
in which case $g(z) = 
(a_0 + a_1 z + \cdots + a_{p-1} z^{p-1})/(1-z^p)$
is rational and uniquely \rrlce.
%
More generally, any rational function which is regular on the Riemann
sphere minus the unit circle and whose poles are simple is uniquely \rrlce;
this follows from Theorem~\ref{thm_BWDuniqueRRL} below
(we shall see that one can even afford for an infinite set of ``poles'' on the unit
circle---we use quotation marks because the function is then no longer rational).
Notice that we restrict ourselves to simple poles because we consider only
the case of bounded Taylor coefficients.

\end{exa}

We emphasize that a holomorphic function~$g$ with bounded Taylor coefficients
may have no \rrlc at all, independently of whether the unit circle is a natural
boundary or not.
For instance, if the sequence of Taylor coefficients of~$g$ at the origin tends
to~$0$, then the only right limit of~$g$ is $b_n\equiv0$ and $g$ cannot be
\rrlce unless $g(z)\equiv0$;
the previous example also shows that no polynomial is \rrlce except
the trivial one.
Observe also that if two holomorphic functions of~$\D$ differ by a function~$h$ which
is holomorphic in a disc $\{\, |z|<R \,\}$ with $R>1$, then they have the same
right limits;
for instance, for any such~$h$, the function $h(z) + (1-z)\ii$ has only one
right limit, the constant sequence $b_n \equiv 1$,
but only when $h(z)\equiv 0$ is this right limit a renascent one.


Notice that with the usual analytic continuation it may happen that, for a given $g\in\gO(\D)$,
there are several arcs through which analytic continuation is possible but leads
to different results. 
Think \eg of $(1+z)^{1/2}(1-z)^{-1/2}$.
However, in view of Proposition~\ref{proprrlalternative}, such examples are not \rrlce.

Note also that, given $r\ge1$ and $g\in\gO(\D)$ divisible by~$z^r$, one has
\[
\text{$h$ is an \rrlc of $g$} 
\ens\Longrightarrow\ens
\text{$z^{-r}h(z)$ is an \rrlc of $z^{-r}g(z)$}
\]
as a consequence of~\eqref{eqlimankj},
but the converse is not necessarily true: 
it may be that $z^{-r}g(z)$ is \rrlce but not~$g$ itself 
(think \eg of $g(z) = z(1-z)\ii$).\footnote{%
But if $g(z)$ is divisible by~$z$ and $z\ii g(z)$ admits as \rrlc
$h(z) = -\dst\sum_{n<0} b_n z^n$,
then $b_{-1} + g(z)$ admits as \rrlc $b_{-1} + z h(z)$.}

\subsection{Functions with values in complex Banach spaces}

Part of the theory can be extended to analytic functions of the disc
with values in an arbitrary complex Banach space. Let $B$ a complex Banach space, and 
 $(a_n)_{n \geq 0}$ a bounded sequence of elements of $B$.
Then the definition of right limit (Definition \ref{D:RL})  and renascent right limit (Definition \ref{defrrl}) of $(a_n)_{n \geq 0}$ 
still makes sense, and the power series 
$$f(z) := \sum a_n z^n$$ 
defines a function holomorphic in $\D$ with values in $B$, 
and one can ask whether $f$ is rrl-continuable and what its rrl-continuations are, since Definition \ref{defrrl}
still makes sense (the only difference being that in the infinite-dimensional case it is not necessarily true that 
any bounded sequence of coefficients $(a_n)_{n \geq 0}$ has at least one right limit).

In this more general context, rrl-continuability still has consequences in terms of natural boundaries. 
In particular, one has the following weaker form of 
Proposition~\ref{proprrlalternative} where
``strong natural boundary" is replaced by ``natural boundary":

\begin{prop}
Let $g$ be an \rrlce function with values in a complex Banach space $B$. Then
\begin{enumerate}[(i)]
\item
either there is an arc of the unit circle through which $g$ admits analytic
continuation; then $g$ is uniquely \rrlce and all the analytic continuations
of~$g$ through arcs of the unit circle match and coincide with the
\rrlc of~$g$;
\item
or the unit circle is a natural boundary for~$g$.

If $g$ is polygenous, then the unit circle is a natural boundary for~$g$.
\end{enumerate}
\end{prop}

\begin{proof}
The proof proceeds exactly as the proof of Proposition 
\ref{proprrlalternative}, using instead of 
Theorem \ref{thmBrSi} the weaker form (\cite{BrSi}, Theorem~1.3).
In fact, the proof of (\cite{BrSi}, Theorem~1.3) generalizes verbatim, 
since it uses only the maximum
principle and Vitali's convergence theorem, which still hold for
vector-valued holomorphic functions (see \eg \cite{HP}).  
\end{proof}

On the other
hand, the statement for strong natural boundaries does not appear to immediately generalize; 
in fact, the proof of Theorem \ref{thmBrSi} (\cite{BrSi}, Theorem~1.4) ultimately uses the
existence of radial limits for functions in the Hardy space~$H^1$,
which does not hold in general.

\section{Dynamical examples}

\subsection{Power series generated by dynamical systems}

A first interesting class of power series which arises in connection with
dynamical systems is as follows:

\begin{Def}
Let $E$ be a metric space and $T\col E\to E$ be a continuous map. Given $\ga\in
E$, we consider its orbit $( \ga_k )_{k\ge0} = \big( T^k(\ga) \big)_{k\ge0}$.
Then, 
for any complex Banach space~$B$ 
and for any bounded function $\ph \col E \to B$, 
we say that the sequence
\[
a_k \defeq \ph\big(T^k(\ga)\big), \qquad {k\ge0}
\]
is \emph{generated by the dynamical system~$T$} (in that situation $\ph$ is
called an \emph{observable}).
\end{Def}


To determine the right limits of the power series 
$\sum_{k=0}^\infty \ph\big(T^k(\ga)\big) z^k$,
one may try to determine first the right limits of the orbit~$(T^k(\ga))$ itself
(and then to exploit continuity or discontinuity properties of the
observable~$\ph$).
The following result is a generalisation of \cite[Lemma~2.4]{BMS}.


\begin{lem}   \label{lemrlfullorb}
The right limits of the orbit $\big(T^k(\ga)\big)_{k\ge0}$ are exactly
the full orbits of~$T$ which are contained in the $\om$-limit set $\om(\ga,T)$.
\end{lem}


\begin{proof}
Recall that the full orbits of $T$ are the two-sided sequences $(\ga_n)_{n\in\Z}$ of~$E$ such that $\ga_{n+1} = T(\ga_n)$
for all $n\in\Z$.
Suppose $(\xi_n)_{n \in\Z}$ is a right limit of $(T^k(\gamma))_{k \geq
  0}$.
Then there exists a sequence $k_j \to \infty$ such that, for each~$n$,
$T^{n+k_j}(\gamma) \to \xi_n$, hence each~$\xi_n$ belongs to $\omega(\gamma, T)$.  
Moreover, by continuity of~$T$, for each~$n$ we have $T^{n+1+k_j}(\gamma) = T(T^{n+k_j}(\gamma))
\to T(\xi_n) = \xi_{n+1}$, hence $(\xi_n)_{n\in\Z}$ is a full orbit.

Conversely, suppose $(\xi_n)_{n\in\Z}$ is a full orbit of $T$ contained in
$\omega(\gamma, T)$. 
By continuity of~$T$, for each $j \geq 1$ we can choose a
neighborhood~$U_j$ of~$\xi_{-j}$ such that, if $x \in U_j$, then
$d(T^k(x), T^k(\xi_{-j})) \leq 1/j$ for all $k \in \{0, 1, \dots, 2j\}$.
Since each~$\xi_{-j}$ belongs to $\omega(\gamma, T)$, we can choose an
increasing sequence $k_j \to \infty$ such that $T^{k_j}(\gamma)$
belongs to~$U_j$ for each~$j$.
Now fix $n \in \mathbb{Z}$; for any $j \geq |n|+1$, we have
$0<n+j<2j$, thus
\[
d(T^{n+k_j+j}(\gamma), \xi_{n})  =
d\big( T^{n+j}(T^{k_j}(\ga)), T^{n+j}(\xi_{-j}) \big) \leq 1/j,
\]
whence $\lim_{j \to \infty} T^{n+k_j+j}(\gamma) = \xi_n$ 
and the claim is proven.
%
%
\end{proof}


One finds in \cite[Theorem~2]{BMS} an example of a sequence
$(a_k)_{k\ge0}$ of the form $\big(\ph\big(T^k(\ga)\big)\big)_{k\ge0}$
which has uncountably many renascent right limits:
its generating series is highly polygenous
(in that example the observable~$\ph$ is continuous but the dynamics~$T$ is a
non-invertible map of a compact interval of~$\R$;
the non-invertibility helps construct a huge set of full orbits).


\begin{exa}
The arithmetic example due to Hecke 
\beglabel{eqdefHeckeex}
\gHecke(z) = \sum_{k=1}^\infty \{k\th\} z^k, \qquad z\in\D,
\elabel
where $\th\in\R\setminus\Q$ and $\{\,\cdot\,\}$ denotes the fractional part,
was shown to have a strong natural boundary in \cite{BrSi}.
This can be viewed as a series generated by the translation $x \mapsto x+\th$ on
$\R/\Z$ for a discontinuous observable.
We shall see in Proposition~\ref{propgHecke} that~$\gHecke(z)$ is uniquely
\rrlce and that $z\ii\gHecke(z)$ has exactly two \rrlcs.
\end{exa}

\subsection{The case of symbolic dynamics}

Another class of examples arises from symbolic dynamics:
if $E = \bigcup_{k = 1}^n P_k$ is a partition of the phase space in a finite
number of sets,
we can define the piecewise constant observable
\[
\ph(x) \defeq  c_k \quad\text{for $x \in P_k$},
\]
for some choice of complex constants $c_1, \ldots, c_n$. 
Then, given a point $x$, the corresponding sequence generated by a dynamical
system $T\col E\to E$ is called \emph{itinerary} of $x$:
\[
\textbf{itin}(x) \defeq \big( \ph\big(T^k(x)\big) \big)_{k\ge0}.
%
%
\]
A powerful application is Milnor--Thurston's \emph{kneading theory} \cite{MT}.
Let $T : [0,1] \to [0,1]$ be a continuous, unimodal map, with $T(0) = T(1) = 0$ and a
critical point $c\in(0,1)$ which we assume non-periodic for simplicity;
we consider the piecewise constant observable~$\ph$ which takes the value~$1$ on
$[0,c]$ and~$-1$ on $(c,1]$.
The \emph{kneading sequence} $(\epsilon_k)_{k\ge0}$
%
%
of~$T$ is defined to be the itinerary of~$c$.
The \emph{kneading determinant} is the power series
\[
D(z) \defeq 1 + \sum_{k\ge1} \epsilon_1\cdots\epsilon_k z^k.
\]
One of the applications of the kneading determinant is to find the topological
entropy of the map. Indeed, if $s$ is the smallest positive real zero of~$D(z)$, then the
entropy of $T$ equals $- \log s$ (\cite{MT}, Theorem 6.3).

\begin{exa}
As an example, consider $T(z) \defeq z^2+c$ with $c$ the Feigenbaum parameter
($c \cong -1.401155189\ldots$).
Then its kneading determinant is known to be
\[
D(z) = \sum_{k = 0}^\infty (-1)^{\tau_k} z^k,
\]
where $\tau \defeq (01101001\dots)$ is the \emph{Thue-Morse sequence} generated
by the substitution $0 \to 01$, $1 \to 10$, starting with $0$.  
Notice that, by the defining relation of $\tau$, it is not hard to prove that
\[
D(z) = \prod_{m = 0}^\infty \big( 1-z^{2^m} \big)
\] 
(from which it follows that the entropy of~$T$ is $0$).
One can check that $D(z)$ has precisely two renascent right limits, hence the unit
circle is a strong natural boundary.
\end{exa}

A thorough investigation of the applications to symbolic dynamics will be the
object of a forthcoming article.

\subsection{Circle maps and the \rrlcy of Hecke's example}

The following result is a variant of a theorem proved in \cite{BrSi} and used
there to show that Hecke's example has a strong natural boundary on the unit circle.
It deals with the series generated by a dynamical system on the circle, with a
special kind of observable:
\begin{thm}	\label{thmVariantHecke}
%
Let  $f\col\T\to\T$ be a homeomorphism of the circle $\T \defeq \R/\Z$, and 
$x^*\in\T$ a point with dense forward orbit under $f$. 
%
Moreover, let  $B$ be a complex Banach space, $\ph \col \T \to B$ a bounded function, and 
$\De\subset\T$ a subset of the torus with empty interior.
Assume that:
\begin{itemize}
\item
 $\ph$ is continuous on $\T\setminus\De$,
\item
each point of~$\De$ is a point of discontinuity for~$\ph$ at which right and
left limits exist and $\ph$ is either right- or left-continuous.
\end{itemize}
Then the function
\[
g(z) \defeq \sum_{k=0}^\infty \ph\big(f^k(x^*)\big) z^k, 
\qquad z\in \D
\]
has the following properties:
\begin{enumerate}[(i)]
\item
If $f^k(x^*) \notin\De$ for all $k>0$, then $g$ is \rrlce.
\item
If $f^k(x^*) \notin\De$ for all $k\ge0$ and there exists $n<0$ such
that $f^n(x^*)\in\De$, then $g$ has at least two different \rrlcs.
\end{enumerate}
\end{thm}


\begin{proof}
Let us use the notation 
\[
y_j \xrightarrow{>} y^*,
\qquad\text{resp.}\quad
y_j \xrightarrow{<} y^*,
\]
if $(y_j)_{j\ge1}$ is a sequence and $y^*$ is a point in~$\T$ for which there
exist lifts $(\ti y_j)_{j\ge1}$ and $\ti y^*$ in~$\R$ such that $\lim_{j\to\infty}\ti y_j
= \ti y^*$ and, for $j$ large enough, $\ti y^*<\ti y_j<\ti y^*+\demi$,
resp.\ $\ti y^*-\demi<\ti y_j<\ti y^*$.
We set
\[
x_n \defeq f^n(x^*), \qquad n\in\Z
\]
and notice that, by the density of $\{x_k\}_{k\ge0}$ in~$\T$, for every
$y^*\in\T$ one can find increasing integer sequences $(k_j^+)_{j\ge1}$ and
$(k_j^-)_{j\ge1}$ such that $x_{k_j^\pm}\xrightarrow{\gtrless} y^*$.

\medskip

Suppose first that $f^k(x^*) \notin\De$ for all $k>0$.
Let us choose an increasing integer sequence $(k_j)_{j\ge1}$ such that 
$x_{k_j} \xrightarrow{\epsilon} x^*$ with $\epsilon$ standing for~`$>$',
unless $x^*\in\De$ and $\ph$ is left-continuous at~$x^*$, 
in which case $\epsilon$ stands for~`$<$'.
Then, for each $n\in\Z$, 
$x_{n+k_j} = f^n(x_{k_j}) \xrightarrow{\epsilon_n} f^n(x^*) = x_n$
with $\epsilon_n$ standing for~`$>$' or~`$<$' according as $f^n$ preserves or
reverses orientation,
and $b_n \defeq \lim_{j\to \infty} \ph(x_{n+k_j})$ exists because $\ph$ has both left and right limits at~$x_n$.
Now, for $n>0$, we have $x_n\notin \De$, hence $b_n = \ph(x_n)$;
for $n=0$, we also have $b_0 = \ph(x_0)$ even if $x_0=x^*\in\De$ thanks to our
choice of $(k_j)$;
therefore we have found a renascent right limit for $(\ph(x_k))_{k\ge0}$.

\medskip

Suppose now that $f^k(x^*) \notin\De$ for all $k\ge0$ and that one can pick
$\ell>0$ such that $f^{-\ell}(x^*)\in\De$.
Let us choose increasing integer sequences $(k_j^+)_{j\ge1}$ and
$(k_j^-)_{j\ge1}$ such that $x_{k_j^\pm}\xrightarrow{\gtrless} x_{-\ell}$.
For each $n\in\Z$, we have $x_{n+\ell+k_j^\pm} \xrightarrow{\epsilon^\pm_n} x_n$,
with $\epsilon^\pm_n$ depending on whether $f^{n+\ell}$ preserves or reverses orientation,
and $b^\pm_n \defeq \lim_{j\to \infty} \ph(x_{n+\ell+k_j^\pm})$ exists because $\ph$ has left and right limits at~$x_n$.
For $n\ge0$, $\ph$ is continuous at~$x_n$, thus $b^+_n = b^-_n = \ph(x_n)$, but for
$n=-\ell$ we have 
\[
b^+_{-\ell} = \lim_{x\xrightarrow{>}x_{-\ell}\hspace{-.7em}} \ph(x)
\; \neq \;
b^-_{-\ell} = \lim_{x\xrightarrow{<}x_{-\ell}\hspace{-.7em}} \ph(x),
\]
which means that we have two different renascent right limits.
\end{proof}


We now discuss the \rrlcy of Hecke's example.

\begin{prop} 	\label{propgHecke}
Let us fix $\th\in\R\setminus\Q$ and denote the fractional part function by
$\{\,\cdot\,\}$.
Then the function 
\[ \gHecke(z) \defeq \sum_{k=1}^\infty \{k\th\} z^k \]
 has a unique \rrlc,
which is
\beglabel{eqrrlcHecke}
\gHecke^-(z) = - \sum_{n<0} \{n\th\} z^n
= \gHecke(z\ii) + (1-z)\ii, \qquad z\in\E.
\elabel
Moreover, the function $z\ii \gHecke(z) = \sum_{k=0}^\infty \{(k+1)\th\} z^k$
has exactly two \rrlcs, namely
\beglabel{eqrrlcziiHecke}
-\sum_{n<0} \{(n+1)\th\} z^n = z\ii\gHecke^-(z)
\quad \text{and}\quad
-z\ii + z\ii\gHecke^-(z),
\elabel
and the unit circle is a strong natural boundary both for~$\gHecke(z)$ and
$z\ii\gHecke(z)$.
\end{prop}


\begin{proof}
Hecke's example falls into case~(i) of Theorem~\ref{thmVariantHecke}:
denoting by $\pi\col\R\to\T$ the canonical projection, we can define the
homeomorphism~$f$ by $f\circ\pi(\ti x) = \pi(\ti x+\th)$ and the discontinuous
observable~$\ph^+$ by $\ph^+\circ\pi(\ti x) = \ti x-\lfloor\ti x\rfloor$, then
$\gHecke(z) = \sum_{k=0}^\infty a_k z^k$
with $a_k \defeq \ph^+\big(f^k(\xHecke)\big)$ and $\xHecke \defeq \pi(0)$, and $\De$ is
reduced to~$\{\xHecke\}$ in that case, with $\ph^+$ right-continuous.
Therefore, since $\th$ is irrational, the forward orbit of $\xHecke$ does not hit $\xHecke$ again, hence 
$\gHecke$ has at least one \rrlc.

To show its uniqueness, we observe that $\ph^+ \col \T \to [0,1)$ is a
section of $\pi \col \R \to \T$ and the restriction of the canonical
projection $\pi^+ \col [0,1) \to \T$ is a right inverse for~$\ph^+$.
Let us thus consider a renascent right limit $\un b = (b_n)_{n\in\Z}$ of
$(a_k)_{k\ge0}$,
and set $y_n \defeq \pi(b_n)$.

The continuity of~$\pi$ entails that $(y_n)$ is a renascent right
limit of $\big( \pi(a_k) \big) = \big( f^k(\xHecke) \big)$,
thus $y_n = f^n(y_0)$ by Lemma~\ref{lemrlfullorb}.
But $y_0 = \xHecke$ because the right limit is renascent, hence $y_n
\in \T\setminus\{ \xHecke \}$ for all $n\in\Z^*$.
Since we already knew that $b_n \in [0,1]$ (because $a_k\in[0,1)$) and
$\pi(b_n) = y_n$, we deduce that, for $n\in\Z^*$, $b_n\in(0,1)$ and
$b_n=\ph^*(y_n)$.
We thus obtain~\eqref{eqrrlcHecke}
(the representation of~$\gHecke^-$ as $\gHecke(z\ii) + (1-z)\ii$ stems from
$\ti x \in\R\setminus\Z \;\Longrightarrow \;-\{-\ti x\} = \{\ti x\}-1$).

On the other hand, the function $g(z) \defeq 
\sum_{k=0}^\infty \{(k+1)\th\} z^k$
falls into case~(ii) of Theorem~\ref{thmVariantHecke}. 
Indeed, the only difference with the previous case is the initial condition,
$x^*_g = \pi(\th)$. 
Therefore, $g$ has at least two \rrlcs.
By Theorem~\ref{thmBrSi}(ii), it follows that the unit circle is a strong
natural boundary for~$g$, and thus also for~$\gHecke$.

Arguing as above, we see that any renascent right limit of~$g$ is of the form
$\un b = (b_n)_{n\in\Z}$ with $y_n \defeq \pi(b_n) = f^n(x^*_g)$; now $y_n \in
\T\setminus\{ \pi(0) \}$ only for $n\in\Z\setminus\{-1\}$, while $y_{-1} = \pi(0)$. 
Since $b_n \in [0,1]$ and $\pi(b_n) = y_n$ for all $n\in\Z$, there are only two possibilities:
$b_{-1} = 0$ or~$1$ and $b_n = \ph^+(y_n) = \ph^+\big( f^n(x^*_g)
\big)$ for $n\in\Z\setminus\{-1\}$.
Both cases are possible (if not there would be only one renascent right limit),
 thus we find the two \rrlcs indicated in~\eqref{eqrrlcziiHecke}.
%
%
%
%
\end{proof}


\begin{cor}
With the same assumptions and notations as in Proposition~\ref{propgHecke}, 
defining 
\beglabel{eqdefgHga} 
\gHe\ga(z) \defeq \sum_{k=0}^\infty \{\ga+k\th \} z^k,
\qquad z\in\D
\elabel
for every $\ga\in\R\setminus(\Z+\th\Z)$, one gets a unique \rrlc for~$\gHe\ga$:
\beglabel{eqrrlcgHga}
\gHe\ga^-(z) = - \sum_{n<0} \{\ga+n\th \} z^n = 
\gHe{-\ga}(z\ii) + z(1-z)\ii + \{\ga\}, 
\elabel
and the unit circle is a strong natural boundary.
\end{cor}

\begin{proof}
The existence of the \rrlc is guaranteed by Theorem~\ref{thmVariantHecke}(i),
exactly as in the proof of Proposition~\ref{propgHecke}.
One finds that~\eqref{eqrrlcgHga} is the only possible \rrlc by following the
same lines.

Since~$0$ is an accumulation point of the sequence 
$(a_k)_{k\ge0} = \big( \{\ga+k\th\} \big)_{k\ge0}$ of $(0,1)$, we can
find a right limit $(b_n)_{n\in\Z}$ such that $b_0=0$, 
and $b_n \in [0,1]$ for every $n\in\Z$.
With the same notations for $\pi$, $f$ and~$\ph^+$ as in the proof of
Proposition~\ref{propgHecke}, since $\pi$ is continuous, it maps $(b_n)_{n\in\Z}$ onto a
right limit $(y_n)_{n\in\Z}$ of the sequence 
$\big(\pi(a_k)\big)_{k\ge0} = \big(f^k\circ\pi(\ga)\big)_{k\ge0}$, 
which is necessarily of the form $y_n = f^n(y_0)$ (by Lemma~\ref{lemrlfullorb}).
Since $y_0 = \pi(0)$, we see that, for $n\in\Z^*$, $y_n \in
\T\setminus\{\pi(0)\}$ and $\pi(b_n) = y_n$, hence $b_n\in(0,1)$ and $b_n = \ph^+(y_n)$.

We just obtained that $(\{n\th\})_{n\in\Z}$ is a right limit of~$\gHe\ga$.
By virtue of Proposition~\ref{propgHecke}, this right limit is not
reflectionless on any arc, Theorem~\ref{thmBrSi}(i) thus implies that the unit
circle is a strong natural boundary.
\end{proof}


\begin{rem}
There is a relationship between the arithmetical properties of~$\th$ and the
functions~$\gHe\ga$: namely, for any $0\le \ga_1 < \ga_2 <1$ one has 
\[
\gHe{\ga_1}(z) - \gHe{\ga_2}(z) - \frac{\ga_2-\ga_1}{1-z} =
\sum_{k\in\cN(\ga_1,\ga_2)} z^k,
\]
where $\cN(\ga_1,\ga_2)$ is the set of visiting times of the sequence
$\big( \{k\th\} \big)_{k\ge0}$ in $[1-\ga_2,1-\ga_1)$, that is
$\cN(\ga_1,\ga_2) \defeq \{\, k\ge0 \mid k\th \in [-\ga_2,-\ga_1)+\Z\,\}$.
\end{rem}


\begin{rem}
Let $\th\in\R\setminus\Q$ and $\ga\in\R$. Replacing the fractional
part function with $t\in\R \mapsto \abs{t}_\Z \defeq \dist(t,\Z)$, we
get an example which looks similar:
\[
g_\ga^+(z) \defeq \sum_{k=0}^\infty \abs{\ga+k\th}_\Z \, z^k,
\qquad z \in \D.
\]
However, the corresponding observable is continuous (it is the
function $\ph \col \T \to \R$ defined by $\ph\big(\pi(\ti x)\big) = \abs{\ga+\ti x}_\Z$), 
hence the study of~$g_\ga^+$ requires completely different techniques. We shall 
prove unique rrl-continuability of~$g_\ga^+$ in Corollary~\ref{corggaHaman}.

\end{rem}


\section{Poincar\'e simple pole series and generalised continuation}

The second half of this article is dedicated to what is probably the simplest
non-trivial situation in which one might wish to test the notion of \rrlcy.
Namely, we will prove that every Poincar\'e simple pole series is uniquely rrl-continuable
(Theorem~\ref{thm_BWDuniqueRRL}); this result was announced
without proof in \cite[Appendix~A.2]{BMS}.

\subsection{Poincar\'e simple pole series and \rrlcy} \label{S:PSP}

We will use the same notations as in Definition~\ref{definnout} for $\PP$, $\D$
and~$\E$,
and denote by~$\S$ the unit circle viewed as a subset of $\C\subset\PP$.

For any real number $\nu>0$ and complex Banach space
$\big(B,\norm{\,\cdot\,}\big)$, we denote by $\enSB$ the set of all
functions $\rho \col \S \to B$ such that the family $\big(
\norm{\rho(\la)}^\nu \big)_{\la\in\S}$ is summable, \ie
\[ \dst\sum_{\la\in\S}\norm{\rho(\la)}^\nu < \infty. \]
Given $\rho\in\enSB$, its \emph{support} is the set (finite or countably infinite)
\[
\supp\rho \defeq \{ \la \in \S \mid \rho(\la) \neq 0 \}. 
\]
%

\begin{Def} \label{D:PSP}
  Let $\rho \in \elSB$. The \emph{Poincar\'e simple pole series}
  (PSP-series for short) associated with~$\rho$ is the $B$-valued
  function $\Sig(\rho)$ defined by
\beglabel{eqdefSigrho}
\Sig(\rho)(z) \defeq \sum_{\la\in\S} \frac{\rho(\la)}{z-\la},
\qquad z \in \PP\setminus\ov{\supp\rho}.
\elabel

The series \eqref{eqdefSigrho} converges normally on every compact
subset of $\PP\setminus\ov{\supp\rho}$, thus in particular it defines
a $B$-valued holomorphic function on $\D \cup \E$.  We shall call
respectively
\emph{inner} and \emph{outer} PSP-series the restrictions of $\Sigma(\rho)$ to the inside and outside the unit disc, 
and denote them by
\[
\Sig(\rho)^+ \defeq \Sig(\rho)_{|\D} \in \gO(\D, B), \qquad
\Sig(\rho)^- \defeq \Sig(\rho)_{|\E} \in \gO(\E, B).
\]
We say that the outer PSP-series $\Sig(\rho)^-$ is \emph{associated} with the
inner PSP-series $\Sig(\rho)^+$.
\end{Def}


The class of inner PSP-series is an interesting class
of functions for which we will prove unique \rrlcy.
It clearly contains the rational functions which are regular on $\PP\setminus\S$
and whose poles are simple, but we are more interested in the case where $\S$ is
a natural boundary.
Our terminology is motivated by Poincar\'e's article \cite{PoinF} (see also
\cite{PoinA}), where he studies this kind of series;
assuming that the support of~$\rho$ is dense in~$\S$, Poincar\'e proves that the
unit circle is a natural boundary for $\Sig(\rho)^\pm$ and he discusses the
relationship between the two functions.
Later Borel, Wolff and Denjoy studied such series, considering also more general
distributions of poles~$\la$ (not restricted to lie on~$\S$).
The subclass of PSP-series $\Sig(\rho)$ with $\supp\rho$ contained in the set of roots of unity
was studied in \cite{MS1} for dynamical reasons, in relation with small divisor
problems.


Any inner PSP-series uniquely determines the associated outer PSP-series,
because $\rho$ and thus $\Sig(\rho)^-$ are uniquely determined by $\Sig(\rho)^+$:

\begin{lem}	\label{lemlimrad}
Let $B$ be a complex Banach space and $\rho \in \elSB$. Then 
\[
\rho(\la) = \lim \ (z-\la) \,\Sig(\rho)^+(z)
\quad \text{as $z\to\la$ radially}
\]
for every $\la\in\S$.
Therefore the map $\rho \in \elSB \mapsto \Sig(\rho)^+ \in \gO(\D, B)$ is
injective.
\end{lem}

\begin{proof}
Given $\la^*\in\S$, we can write
\[
(z-\la^*) \,\Sig(\rho)^+(z) - \rho(\la^*) =
\sum_{\la \in \S\setminus\{\la^*\}} \rho(\la) \frac{z-\la^*}{z-\la\,},
\qquad z\in\D.
\]
For each $\la \in \S\setminus\{\la^*\}$, we have
$z\in[0,\la^*]$ $\Rightarrow$ $\abs*{z-\la^*} < \abs*{z-\la}$,
thus $\norm*{\rho(\la)\frac{z-\la^*}{z-\la\,}}
\le \norm*{\rho(\la)}$.
Now, since $\frac{z-\la^*}{z-\la\,} \to 0$ as $z\to\la^*$,
the result follows by dominated convergence.
\end{proof}

\begin{rem}
In fact, one even has
\[
\rho(\la) = \lim\  (z-\la) \,\Sig(\rho)^+(z)
\quad \text{as $z\to\la$ non-tangentially}
\]
for every $\la\in\S$.
\end{rem}


Lemma~\ref{lemlimrad} shows that, if $\la \in \supp\rho$, the
function~$\Sig(\rho)^+$ is not bounded on the ray $[0,\la]$, hence $\la$ is
necessarily a singular point of the function. This entails a dichotomy:
\begin{enumerate}[(a)]
\item
either $\supp\rho$ is not dense in~$\S$; then $\PP\setminus\ov{\supp\rho}$ is
connected and $\S\setminus\ov{\supp\rho}$ is a countable union of open arcs of
the unit circle, in the neighbourhood of which $\Sig(\rho)$ is holomorphic; we
can thus view $\Sig(\rho)^+$ and~$\Sig(\rho)^-$ as the analytic continuation 
of each other through any of these arcs;
\item
or $\supp\rho$ is dense in the unit circle, $\D$ and~$\E$ are the two connected
components of $\PP\setminus\ov{\supp\rho} = \PP\setminus\S$
and the unit circle is a natural boundary for both $\Sig(\rho)^+$
and~$\Sig(\rho)^-$.
\end{enumerate}

%


%
Our main theorem about inner PSP-series is in terms of the notion of generalised
analytic continuation discussed in Section~\ref{secremindrrlc}:
the outer PSP-series is the unique \rrlc of the inner PSP-series with which it
is associated.
This result was stated in the introduction as
Theorem~\ref{thmPSPuniquerrlce} in the case where $B=\C$, we restate
it now in full generality:

\begin{thm}	\label{thm_BWDuniqueRRL}
Let $B$ a complex Banach space.
Suppose that $g = \Sigma(\rho)^+ \in\gO(\D,B)$ is an inner PSP-series,
with $\rho \in \elSB$. Then~$g$ is uniquely \rrlce and
its \rrlc is the associated outer PSP-series $\Sigma(\rho)^-$.
\end{thm}

The proof will be given in Section~\ref{secproofmainthm}.
Together with Proposition~\ref{proprrlalternative}, this immediately yields


\begin{cor}	\label{corBWDrrl} 
Assume that $B=\C$. 
Suppose that the unit circle is a natural boundary for an inner
PSP-series $g = \Sigma(\rho)^+ \in \gO(\D)$ 
(\ie the support of the corresponding $\rho\in\elSC$ is dense in~$\S$).
Then the unit circle is a strong natural boundary for~$g$.
\end{cor}


\subsection{An example based on harmonic analysis}

With the help of Theorem~\ref{thm_BWDuniqueRRL} we easily get

\begin{prop}    \label{propPSPharman}
Let $B$ be a complex Banach space.
Let $\ph \col \R \to B$ be continuous and $1$-periodic and assume that the sequence
$(\hat\ph(j))_{j\in\Z}$ of its Fourier coefficients is absolutely convergent:
\beglabel{ineqcondF}
\sum_{j\in\Z} \abs{\hat\ph(j)} < \infty.
\elabel
Then, for any $\th\in\R\setminus\Q$, the sum of the convergent power series
\beglabel{eqdefphnth}
g^+(z) \defeq \sum_{n\ge0} \ph(n\th) z^n, \qquad z\in\D
\elabel
is an inner PSP-series, thus uniquely \rrlce, with associated outer PSP-series given by
\beglabel{eqdefBWDphnth}
g^-(z) \defeq -\sum_{n<0} \ph(n\th) z^n, \qquad z\in\E.
\elabel
Moreover, if $\ph$ is not a trigonometric polynomial, then the unit circle is a
natural boundary for~$g^+$ (a strong one if $B=\C$).
\end{prop}

\begin{proof}
Consider the pairwise distinct points $\la_j = \ee^{-2\pi\I j\th}$, $j\in\Z$,
and define
$\rho \col \S \to\C$ by
\[
\rho(\la_j) = - \la_j \hat\ph(j), \qquad j\in\Z,
\]
and $\rho(\la) = 0$ if $\la$ is not an
integer power of $\ee^{-2\pi\I\th}$; 
then $\rho \in \elSB$ and 
$-\sum \rho(\la) \la^{-n-1} = \sum \hat\ph(j) \ee^{2\pi\I jn\th} = \ph(n\th)$
for every $n\in\Z$,
hence the formula~\eqref{eqdefbnBWD} for the Taylor coefficients at~$0$ of an
inner PSP-series and the coefficients at~$\infty$ of the associated
outer PSP-series yields
$g^\pm = \Sig(\rho)^\pm$.
The rest follows from Theorem~\ref{thm_BWDuniqueRRL}.
\end{proof}


The previous result is reminiscent of Hecke's example: 
the formula~\eqref{eqrrlcHecke} that we obtained for its \rrlc looks like
an echo of~\eqref{eqdefBWDphnth}.
In Hecke's example, however,
the observable~$\ph^+$ violates condition~\eqref{ineqcondF} (since its
Fourier coefficients decay only as $\frac{1}{j}$) and is not continuous.
Here is an example, similar in nature to Hecke's example or its
generalisation~\eqref{eqdefgHga}, but satisfying the assumptions of
Proposition~\ref{propPSPharman}:

\begin{cor}   \label{corggaHaman}
  We set, for any $t\in\R$, $\abs{t}_\Z \defeq \dist(t,\Z)$ the
  distance from~$t$ to the closest integer.
Let $\th \in \R\setminus\Q$.
Then, for each $\ga\in\R$, the sum of the convergent power series
\beglabel{eqggaharman}
g_\ga^+(z) \defeq \sum_{n\ge0} \abs{\ga+n\th}_\Z \, z^n, \qquad z\in\D
\elabel
is an inner PSP-series, thus uniquely \rrlce, with \rrlc given by
\beglabel{eqoutggaharman}
g_\ga^-(z) \defeq -\sum_{n<0} \abs{\ga+n\th}_\Z \, z^n 
= -g_{-\ga}^+(z\ii) + \abs{\ga}_\Z, \qquad z\in\E
\elabel
and the unit circle is a strong natural boundary for~$g_\ga^+$.
\end{cor}

\begin{proof}
The Fourier coefficients of $\ph_\ga(t) \defeq \abs{\ga+t}_\Z$ are
$O(1/j^2)$ and $\ph_\ga(-t) = \ph_{-\ga}(t)$.
\end{proof}

\subsection{Poincar\'e simple pole series and monogenic regularity} \label{S:mon}

In this section we will compare the notion of rrl-continuability
developed so far to the property of \emph{monogenic regularity}
introduced by Borel \cite{Bo}. 
%
Monogenic regularity is an alternative way to generalised analytic
continuation, and spaces of monogenic functions enjoy
quasianalyticity properties that we will recall in this section.
We will prove that, at least for certain PSP-series, the two theories
of continuation overlap, since the outer function is the continuation
of the inner function according to both definitions.


Let~$B$ be a complex Banach space.
Given $K' \subset K \subset \PP$ and a linear space~$E$ of $B$-valued
functions defined on~$K$,
we say that~$K'$ is a \emph{uniqueness set} for~$E$ if
the only function of~$E$ which vanishes on~$K'$ is the zero function,
and we say that~$E$ is \emph{$\gH^1$-quasianalytic relatively to $K'$} if any subset of~$K'$
of positive one-dimensional Hausdorff measure is a uniqueness set
for~$E$.


For a closed subset~$K$ of~$\PP$, we denote by $\gO(K,B)$ the space
of all $B$-valued functions which are continuous on~$K$ and
holomorphic in the interior of~$K$;
and we denote by $\gC^1_{hol}(K, B) \subset \gO(K,B)$ the Banach space of all $B$-valued
functions which are $C^1$-holomorphic on~$K$ (\ie Whitney-differentiable in the complex
sense---see \cite{MS1} or \cite{CMS} for a precise definition).

Let $(K_j)_{j\ge0}$ be a monotonic non-decreasing sequence of closed
subsets of~$\PP$.  
The space of Borel monogenic functions $\gM\bigl( (K_j), B)$ is
defined as the projective limit
\[ \gM\bigl( (K_j), B) \defeq \lim_{\leftarrow} \gC^1_{hol}(K_j, B). \]
It is proven in \cite{MS2} that, for certain sequences $(K_j)$, the
space $\gM\bigl( (K_j), B)$ is $\gH^1$-quasianalytic relatively to $K
\defeq \bigcup K_j$. 
The next proposition shows that PSP-series enjoy this quasianalyticity
property, at least for $\rho \in \qlSB$.


\begin{thm}	\label{thmmonog}
Let $\rho \in \qlSB$.
Then there exists an increasing sequence $(K_j)$ of compact subsets of~$\PP$
such that 
\begin{itemize}
\item
  the set $K\defeq \bigcup K_j$ has its complement contained in
  $\ov{\supp\rho} \subset \S$ and of zero Haar measure;
\item
  the function~$\Sig(\rho)$ has a unique continuous extension to~$K$;
\item
  the space of Borel monogenic functions associated with $(K_j)$,
\[
\gM\bigl( (K_j), B \bigr),
\]
contains this extension of~$\Sig(\rho)$ and is $\gH^1$-quasianalytic
relatively to~$K$.
\end{itemize}
\end{thm}


The $\gH^1$-quasianalyticity property means that any function~$g$ of the space
$\gM\bigl( (K_j), B \bigr)$ is determined by its restriction to any subset
of~$K$ which has positive linear Hausdorff measure;
in particular it is determined by its inner restriction $g_{|\D}$.
In the case of $g = \Sig(\rho)$ as in Theorem~\ref{thmmonog},
this yields a totally different way of recovering the outer function
$g_{|\E} = \Sig(\rho)^-$ from the inner function $\Sig(\rho)^+$.
In view of Theorem~\ref{thm_BWDuniqueRRL}, one may wonder whether it is true
that $g_{|\E}$ is the only \rrlc of $g_{|\D}$ for any $g\in\gM\bigl( (K_j), B
\bigr)$.


\begin{proof}[Proof of Theorem~\ref{thmmonog}]
For each $j\ge1$, define 
\[
K_j \defeq \bigg\{\, 
z \in \C \mid \text{for each $\la \in \S$,}\;
\abs{z-\la} \ge   \frac{1}{j} \norm{\rho(\la)}^{1/4} 
\, \bigg\} \cup\{\infty\}.
\]
Each~$K_j$ is a compact subset of the Riemann sphere, with $K_j
\subset K_{j+1}$, and the fact that~$\rho$ is bounded implies that the
complement of $K \defeq \bigcup_{j\ge1} K_j$ is contained
$\ov{\supp\rho}$ and thus in~$\S$; it has zero Haar measure because
$\rho \in \qlSB$.

Let us check that the space $\gM\big((K_j), B\big)$ is $\gH^{1}$-quasianalytic.
Denoting by $\Ga_j\IN \defeq \pa \big(K_j \cap \ov\D \big)$ and $\Ga_j\EX \defeq
\pa \big(K_j \cap \ov\E \big)$,
%
we observe that $\Ga_j\IN$ and~$\Ga_j\EX$ are rectifiable Jordan curves
and, for each~$j$ sufficiently large, 
the set $\Ga_j\IN \cap \Ga_j\EX = K_j \cap \S$ has positive
one-dimensional Hausdorff measure, since
\[
\abs{ \{\, \th \in [0, 2\pi] \mid \ee^{\I \th} \in \S\setminus K_j \,\} } 
\leq \sum_{\la \in \S} 4 \arcsin\bigg(\frac{ \norm{\rho(\la)}^{1/4} } { j} \bigg)
\leq 2\pi \sum_{\la \in \S} \frac{ \norm{\rho(\la)}^{1/4} } { j}.
%
%
%
\]
%
%
Thus, in the language of \cite[Definition~4]{MS2}, each pair
$(\Ga_j\IN, \Ga_j\EX)$ is a nested pair and $K_j =
K\big(\Ga_j\IN,\Ga_j\EX\big)$, so the space $\gM\big((K_j), B\big)$ of monogenic
functions is $\gH^{1}$-quasianalytic by \cite[Corollary~A]{MS2}.

Let us now check that this space contains~$\Sig(\rho)$.  
The normal convergence properties
\[ 
z \in K_j \ens\Rightarrow \ens
\left\{ \begin{aligned}
\sum_{\la \in \S} \frac{\norm{\rho(\la)}}{|z-\la|} 
\leq j \sum_{\la \in \S} \norm{\rho(\la)}^{3/4} < \infty, \\[1ex]
\sum_{\la \in \S} \frac{\norm{\rho(\la)}}{|z-\la|^2} 
\leq j^2 \sum_{\la \in \S} \norm{\rho(\la)}^{1/2} < \infty,
\end{aligned} \right.
\]
which hold for each~$j$,
show that the series $f(z) \defeq \sum_{\la\in\S}
\frac{\rho(\la)}{z-\la}$
and
$f^{(1)}(z) \defeq \sum_{\la\in \S} \frac{\rho(\la)}{(z-\la)^2}$
define continuous functions in~$K$.
Moreover, given $j\ge1$, we have for any $\eps>0$
\begin{align*}
\abs{f(z_2)-f(z_1) - f^{(1)}(z_1)(z_2-z_1)} &= 
\abs{ (z_1-z_2)^2 \sum_{\la \in \S} \frac{\rho(\la)}{(z_1-\la)^2(z_2-\la)}} \\[1ex]
&\le \abs{z_1-z_2}^2 j^3 \sum_{\la \in \S} \norm{\rho(\la)}^{1/4} \le \eps \abs{z_1-z_2}
\end{align*}
as soon as $z_1, z_2 \in K_j$ and $\abs{z_1-z_2}$ is sufficiently
small, hence~$f$ is $C^1$-holomorphic on~$K_j$.
We conclude that $f \in \gM\big((K_j), B\big)$.
Clearly, $\Sig(\rho)$ coincides with the restriction of~$f$ to
the set $\PP\setminus\ov{\supp\rho}$;
the closure of the latter set is~$\PP$ thus, by continuity, $f$ is unique.
\end{proof}


\begin{rem}
If we do not suppose $\rho \in \qlSB$ but only $\rho \in \dlSB$, then
the same construction as above is sufficient to get an extension
of~$\Sig(\rho)$ which belongs to $\gO(K_j,B)$ for each~$j$.
Observe that the spaces $\gO(K_j,B)$ too are $\gH^1$-quasianalytic by
virtue of \cite[Corollary~A]{MS2}.
\end{rem}


\begin{exa}
Let $G(w) = \sum_{m=1}^\infty G_m w^m$ be holomorphic for $\abs{w} <
r_0$ (with values in~$\C$).
A particular case of Theorem~\ref{thmmonog} occurs when solving the
``cohomological equation''
\begin{equation} \label{E:cohom}
f(qw)-f(w) = G(w). 
\end{equation}
One then gets a vector-valued PSP-series, which was studied under the
name ``Borel-Wolff-Denjoy series'' in \cite{MS1}.
Indeed, for any $r<r_0$, the functional equation \eqref{E:cohom}
has solution
\beglabel{eqfundsolncohom}
f_q(w) = \sum_{m=1}^\infty G_m \frac{w^m}{q^m - 1} 
\elabel
in the space $B = \B_r$ of all bounded
holomorphic functions of the disc $\{\, w\in\C \mid \abs{w}<r \,\}$.
Taking into account the dependence on the ``multiplier" $q$ thus defines a map 
$q \in \D \mapsto f_q(\cdot) \in \B_r$, which turns out to be a 
PSP-series with values in $\B_r$ (where the role of the variable~$z$ of
the present article is played by $q$). 
Indeed, one can rewrite \eqref{eqfundsolncohom} as 
\[
f_q = \sum_{\la \in \S} \frac{\rho(\la)}{q-\la}
\]
%
%
%
%
%
%
%
where the map $\rho \col \S \to \B_r$ has its support
contained in the set of all roots of unity and is defined by
\beglabel{eqdefrhocohom}
\rho(\la)(w) \defeq \la \sum_{k=1}^\infty G_{km} \frac{w^{km}}{km} 
\elabel
for $\la$ primitive root of unity of order~$m$ (see also \cite{MS1}).
An easy estimate shows that $\norm{\rho(\la)}$ decays at least
geometrically \wrt~$m$, thus $\rho \in \ell^{1/4}(\S,\B_r)$ and
Theorem~\ref{thmmonog} immediately entails
\begin{prop}
The function $q \in \D \mapsto f_q(\,\cdot\,) \in \B_r$ 
is an inner PSP-series which has an $\gH^1$-quasianalytic monogenic
continuation through the unit circle, given by the sum of the \rhs
of~\eqref{eqfundsolncohom} for $\abs{q}>1$.
\end{prop}


If~$G$ is not a polynomial, then the support of~$\rho$ is dense
in~$\S$
(because $G_m\neq0$ implies that $\rho(\la)\neq0$ for all~$\la$ such
that $\la^m=1$)
and the unit circle is thus a natural boundary for this function.
As a consequence of Corollary~\ref{corBWDrrl}, we also get
\begin{prop}
If $G$ is not a polynomial and $w_*\in\C$ satisfies $0 < \abs{w_*} < r_0$,
then the function $q \in \D \mapsto f_q(w_*) \in \C$ has a strong
natural boundary on the unit circle.
\end{prop}


\begin{proof}
  We set $G^*_n = G_n w_*^n / n$ for every $n\in\N^*$
  and suppose that the unit circle is not a strong natural boundary.
  Since $q \mapsto f_q(w_*) = \Sig\big( \rho(\,\cdot\,)(w_*) \big)^+(q)$ is a scalar inner PSP-series,
  Corollary~\ref{corBWDrrl} entails that the support of the function
  $\la \in \S \mapsto \rho(\la)(w_*)$
  is not dense in~$\S$. In view of~\eqref{eqdefrhocohom}, this says
  that there exists $m_0\in\N^*$ such that the complex numbers
\[
\rho^*_m \defeq \sum_{k\ge1} G^*_{km}, \qquad m\in \N^*,
\]
vanish for $m \ge m_0$.
But, by M\"obius inversion, $\dst G^*_n = \sum_{d\ge1}
\mu(d)\rho^*_{dn}$ for every $n\in\N^*$
(where $\mu(d)$ is the sum of all primitive roots of unity of order~$d$), 
hence $G_n = 0$ for $n\ge m_0$.
\end{proof}
\end{exa}









\section{Proof of Theorem~\ref{thm_BWDuniqueRRL}} \label{secproofmainthm}

Let~$B$ be a complex Banach space. Let $\rho\in \elSB$.
The Taylor expansion of~$\Sig(\rho)^+$ at the origin and the Taylor expansion of~$\Sig(\rho)^-$
at~$\infty$ are easily computed by expanding the geometric series
$\frac{1}{z-\la} = - \la\ii(1-\la\ii z)\ii
= z\ii(1-\la z\ii)\ii$ and permuting sums; one can write the result as
\begin{gather}
\label{eqexpandgh}
\Sig(\rho)^+(z) = \sum_{n\ge0} b_n z^n, \qquad 
\Sig(\rho)^-(z) = -\sum_{n<0} b_n z^n, \\[1ex]
\label{eqdefbnBWD}
b_n \defeq - \sum_{\la\in\S} \rho(\la) \la^{-n-1} \quad \text{for $n\in\Z$.}
\end{gather}


\subsection{The associated outer PSP-series is an \rrlc}	\label{secpfrrlce}


We prove in this section the first part of
Theorem~\ref{thm_BWDuniqueRRL}, namely that~$\Sig(\rho)^-$ is an \rrlc
of~$\Sig(\rho)^+$.
It is enough to find an unbounded integer sequence $(k_j)_{j\ge1}$
such that 
\beglabel{eqrrlbn}
\lim_{j\to\infty} b_{n+k_j}  = b_n \quad \text{for every $n\in\Z$}
\elabel
(indeed, from any such unbounded sequence, one can extract an increasing
sequence for which \eqref{eqrrlbn} still holds, showing that $\un b = (b_n)_{n\in\Z}$ is
a right limit of $(b_k)_{k\ge0}$, with $\Sig(\rho)^+=g_{\un b}^+$ and $\Sig(\rho)^-=g_{\un b}^-$).

We have
$b_{n+k_j}  = - \sum_{\la\in\S} \la^{-k_j} \cdot \rho(\la) \la^{-n-1}$
and $\norm{\la^{-k_j} \cdot \rho(\la) \la^{-n-1}} = \norm{\rho(\la)}$.
Thus, by the dominated convergence theorem, if the sequence $(k_j)_{j\ge1}$ satisfies
\beglabel{eqlakj}
\lim_{j\to\infty} \la^{k_j} = 1
\quad \text{for each $\la\in\supp\rho$}
\elabel
then \eqref{eqrrlbn} holds.
The problem thus reduces to finding an unbounded integer sequence $(k_j)_{j\ge1}$
satisfying~\eqref{eqlakj}. 

\medskip


\emph{\underline{1st case}}

\smallskip


Assume that $\supp\rho$ is contained in the set of all roots of unity.
Then one can take $k_j \defeq j!$,
since $\la^{j!}=1$ for any $j\ge$ order of~$\la$ as a root of unity.

\bigskip


\emph{\underline{2nd case}}

\smallskip


Assume that $\supp\rho$ is not contained in the set of all roots of unity.
We write $\supp\rho = \{\la_1,\la_2,\ldots\}$.
For each $j\ge1$, we set
\[
V_j \defeq \{\, \ee^{2\pi\I\om} \mid 0 \le \om < 1/j \,\} \subset \S
\]
and we consider the ``cells''
\[
W_{\ell_1,\ell_2,\ldots,\ell_j} \defeq 
\ee^{2\pi\I\ell_1/j} V_j \times \ee^{2\pi\I\ell_2/j} V_j 
\times \cdots \ee^{2\pi\I\ell_j/j} V_j 
\subset \S^j
\]
(where $\ee^{2\pi\I\ell_r/j} V_j$ is short-hand for 
$\{\, \ee^{2\pi\I\om} \mid {\ell_r}/{j} \le \om < {(\ell_r+1)}/{j} \,\}$)
for each integer $j$-tuple $(\ell_1,\ell_2,\ldots,\ell_j)$ with 
$0 \le \ell_1,\ell_2,\ldots,\ell_j \le j-1$;
these are $j^j$ cells which cover the torus~$\S^j$.
Now consider the $j^j+1$ points
\[
\La_{j,m} \defeq (\la_1^m,\la_2^m,\ldots,\la_j^m) \in \S^j,
\quad \text{for $m=0,1,\ldots,j^j$}.
\]
Out of them, at least two belong to the same cell, we thus can find $m_j < m_j'$
such that $\La_{j,m_j}$ and~$\La_{j,m_j'}$ belong to the same cell 
$W_{\ell_1,\ell_2,\ldots,\ell_j}$;
this means that
\[
\la_r^{m_j\ministrut}, \la_r^{m_j'} \in \ee^{2\pi\I\ell_r/j} V_j
\]
for all $r=1,2,\ldots,j$. This implies
\beglabel{eqdefkj}
\la_r^{k_j} \in V_j,
\quad \text{where $k_j \defeq m_j' - m_j$}
\elabel
for all $r\le j$.
Keeping $r$ fixed but arbitrary, we thus get $\lim_{j\to\infty}\la_r^{k_j} = 1$.
Therefore we have obtained~\eqref{eqlakj} with the sequence $(k_j)_{j\ge1}$
defined by~\eqref{eqdefkj}.
Now this sequence cannot be bounded because, if it were, \eqref{eqlakj} would
imply that each element of $\supp\rho$ is a root of unity.
This ends the first part of the proof of Theorem~\ref{thm_BWDuniqueRRL}.


\subsection{Uniqueness of the \rrlc for an inner PSP-series} \label{secpfunique}


We now prove the second part of Theorem~\ref{thm_BWDuniqueRRL}, namely
that there is no \rrlc for $\Sig(\rho)^+$ other than $\Sig(\rho)^-$.


\begin{prop}    \label{propScalarImplic}
Suppose that a strictly increasing sequence of positive integers $(k_j)_{j\ge1}$
satisfies
\beglabel{eqlimbnkjpos}
\lim_{j\to\infty}  b_{n+k_j} =  b_n \quad \text{for each $n\ge 0$,}
\elabel
then
\beglabel{eqlimlakj}
\lim_{j\to\infty} \la^{k_j} = 1 \quad \text{for each $\la \in \supp \rho$.}
\elabel
\end{prop}


\begin{proof}[Proof of Proposition~\ref{propScalarImplic}]
Let $(k_j)_{j\ge1}$ be as in the statement.
For each $j\ge1$ we define a map $\de_j \col \S \to B$ by
\[ 
\de_j(\la) := \rho(\la) \la^{-k_j} - \rho(\la)
\quad \text{for each $\la\in\S$.}
\]
We just need to prove that 
$\lim_{j\to\infty} \de_j(\la) = 0$ for each $\la\in\S$.
Notice that 
\beglabel{ineqmajdejla}
\norm{\de_j(\la)} \le 2 \norm{\rho(\la)},
\elabel
hence $\de_j\in\elSB$.
By \eqref{eqexpandgh}--\eqref{eqdefbnBWD}, we have
\[
\Sig(\de_j)^+(z) = \sum_{n\ge0} c_{j,n} z^n
\quad \text{for $z\in\D$}
\]
with $c_{j,n} = - \sum_{\la\in\S} \big(\rho(\la) \la^{-k_j}-\rho(\la)\big) \la^{-n-1} =
b_{n+k_j} - b_n$.
The assumption~\eqref{eqlimbnkjpos} implies $\lim_{j\to\infty} c_{j,n} = 0$ for each
$n\ge0$.
Since $\norm{c_{j,n} z^n} \le 2 \norm{\rho}_{\elSB} \abs{z}^n$,
the dominated convergence theorem yields
\beglabel{limCVsimple}
\lim_{j\to\infty} \Sig(\de_j)^+(z) = 0 
\quad \text{for each $z\in\D$.}
\elabel

Fix $\la^* \in \S$. By the same computation as in the proof of
Lemma~\ref{lemlimrad}, for each $j\ge1$ and $z\in [0,\la^*)$ we have
\begin{multline*}
\norm{ \de_j(\la^*) - (z-\la^*)\Sig(\de_j)^+(z) } =
\norm{ \sum_{\la\neq\la^*} \de_j(\la) \frac{z-\la^*}{z-\la} } \\[1ex]
\le
\eps^*(z) \defeq 2 \sum_{\la\neq\la^*} \norm{\rho(\la) \frac{z-\la^*}{z-\la} }
\end{multline*}
(using~\eqref{ineqmajdejla}).
By dominated convergence, 
\beglabel{eqlimepsst}
\lim\eps^*(z) = 0 \quad \text{as $z\to\la^*$ radially}
\elabel
(because $z\in[0,\la^*]$ $\Rightarrow$ $\abs{z-\la^*} < \abs{z-\la}$, 
thus $\norm*{\rho(\la) \frac{z-\la^*}{z-\la\,}} \le \norm{\rho(\la)}$).

Let $\epsilon>0$. For every $j\ge1$ and $z\in[0,\la^*)$, we have
\beglabel{ineqwritingdejlast}
\norm{\de_j(\la^*)} \le \norm{ (z-\la^*)\Sig(\de_j)^+(z) } + \eps^*(z).
\elabel
By~\eqref{eqlimepsst}, we can choose $z\in[0,\la^*)$ such that
$\eps^*(z) \le \frac{\epsilon}{2}$.
Using~\eqref{limCVsimple}, for that particular~$z$, we can find
$J\ge1$ such that for all $j\ge J$ also the first term in the \rhs
of~\eqref{ineqwritingdejlast} is $\le \frac{\epsilon}{2}$.
This shows that $\lim_{j\to\infty} \de_j(\la^*) = 0$ for every
$\la^*\in\S$, which was the desired conclusion.
\end{proof}


\begin{proof}[Proposition~\ref{propScalarImplic} allows to complete
  the proof of Theorem~\ref{thm_BWDuniqueRRL}]
Suppose that $h(z) = - \sum_{n<0} \be_n z^n$ is an \rrlc of~$\Sig(\rho)^+$.
We must show that~$h$ coincides with~$\Sig(\rho)^-$,
\ie that $\be_n = b_n$ for each $n<0$.

By hypothesis, there is a strictly increasing sequence of positive integers
$(k_j)_{j\ge1}$ such that
$\be_n = \lim_{j\to\infty} b_{n+k_j}$ for each $n<0$ and
$b_n = \lim_{j\to\infty} b_{n+k_j}$ for each $n\ge0$.
Proposition~\ref{propScalarImplic} entails 
\[
\lim_{j\to\infty}\la^{k_j} = 1
 \quad \text{for each $\la \in \supp\rho$,}
\]
therefore, for each $n<0$, 
$b_{n+k_j} = - \sum_{\la\in\S} \rho(\la) \la^{-n-1} \cdot \la^{-k_j}
\xrightarrow[j\to\infty]{} b_n$ by the dominated convergence theorem
(since $\norm{\rho(\la) \la^{-n-1} \cdot \la^{-k_j}} \le \norm{\rho(\la)}$),
whence $b_n=\be_n$.
\end{proof}


\vfill\eject


\noindent {\em Acknowledgements.}
%
%
We thank Gauthier Gidel and Florestan Martin-Baillon for considerably simplifying 
the proof of the second part of Theorem~\ref{thm_BWDuniqueRRL}
with respect to the first version of the paper.
%
%
We thank Viviane Baladi for interesting discussions and encouragement.
We acknowledge the support of the Centro di Ricerca Matematica Ennio de
Giorgi.


\vspace{1cm}



\vspace{1cm}

\noindent
David Sauzin\\[1ex]
CNRS UMI 3483 - Laboratorio Fibonacci \\
Centro di Ricerca Matematica Ennio De Giorgi, \\
Scuola Normale Superiore di Pisa \\
Piazza dei Cavalieri 3, 56126 Pisa, Italy\\
email:\,{\tt{david.sauzin@sns.it}}

\bigskip

\noindent
Giulio Tiozzo\\[1ex]
%
Yale University \\
10 Hillhouse Avenue \\
New Haven, 06511 CT, USA\\
email:\,{\tt{giulio.tiozzo@yale.edu}} \\


\end{document}